\pdfoutput=1
\RequirePackage{ifpdf}
\ifpdf 
\documentclass[pdftex]{sigma}
\else
\documentclass{sigma}
\fi

\numberwithin{equation}{section}

\newtheorem{Theorem}{Theorem}[section]
\newtheorem{Corollary}[Theorem]{Corollary}
\newtheorem{Proposition}[Theorem]{Proposition}
 { \theoremstyle{definition}
\newtheorem{Definition}[Theorem]{Definition}
\newtheorem{Remark}[Theorem]{Remark} }

\usepackage{mathtools}

\def\N{\mathbb N}

\def\1{\mathbf 1}

\usepackage{tikz,tkz-tab, tikz-cd}
\usepackage{difftrees}
\tikzset{
  ncone/.pic={
	\draw (0,0)--(0,0.2);
  }
}

\tikzset{
  nctwo/.pic={
    \draw (0,0)--(0,0.2);
	\draw (0.1,0)--(0.1,0.2);
  }
}

\tikzset{
  nctwoW/.pic={
    \draw (0,0.2)--(0,0)--(0.1,0)--(0.1,0.2);
  }
}

\tikzset{
  nctwoWW/.pic={
    \draw (0,0.2)--(0,0)--(0.2,0)--(0.2,0.2);
  }
}

\tikzset{
  ncthreeWW/.pic={
    \draw (0,0.2)--(0,0)--(0.3,0)--(0.3,0.2);
	\draw (0.2,0)--(0.2,0.2);
  }
}

\tikzset{
  ncthree/.pic={
    \draw (0,0)--(0,0.2);
	\draw (0.1,0)--(0.1,0.2);
	\draw (0.2,0)--(0.2,0.2);
  }
}

\tikzset{
  ncthreeW/.pic={
    \draw (0,0.2)--(0,0)--(0.2,0)--(0.2,0.2);
	\draw (0.1,0)--(0.1,0.2);
  }
}

\tikzset{
  ncfour/.pic={
    \draw (0,0)--(0,0.2);
	\draw (0.1,0)--(0.1,0.2);
	\draw (0.2,0)--(0.2,0.2);
	\draw (0.3,0)--(0.3,0.2);
  }
}
\tikzset{
  ncfive/.pic={
    \draw (0,0)--(0,0.2);
	\draw (0.1,0)--(0.1,0.2);
	\draw (0.2,0)--(0.2,0.2);
	\draw (0.3,0)--(0.3,0.2);
	\draw (0.4,0)--(0.4,0.2);
  }
}

\tikzset{
  ncfourW/.pic={
    \draw (0,0.2)--(0,0)--(0.3,0)--(0.3,0.2);
	\draw (0.1,0)--(0.1,0.2);
	\draw (0.2,0)--(0.2,0.2);
  }
}
\tikzset{
  ncfiveW/.pic={
    \draw (0,0.2)--(0,0)--(0.4,0)--(0.4,0.2);
	\draw (0.1,0)--(0.1,0.2);
	\draw (0.2,0)--(0.2,0.2);
	\draw (0.3,0)--(0.3,0.2);
  }
}

\tikzset{
  nconeinsidetwoWW/.pic={
    \path (0,0) pic {nctwoWW}; \path (0.1,0.1) pic {ncone};
  }
}

\tikzset{
  nconeinsidethreeWW/.pic={
    \path (0,0) pic {ncthreeWW}; \path (0.1,0.1) pic {ncone};
  }
}
\tikzset{
  nconeinsidethreerightWW/.pic={
    \draw (0,0.2)--(0,0)--(0.3,0)--(0.3,0.2);
    \draw (0.1,0)--(0.1,0.2); \draw (0.2,0.1)--(0.2,0.3);
  }
}
\tikzset{
  nconeinsidethreeleftWW/.pic={
    \draw (0,0.2)--(0,0)--(0.3,0)--(0.3,0.2);
    \draw (0.2,0)--(0.2,0.2); \draw (0.1,0.1)--(0.1,0.3);
  }
}
\tikzset{
  nctwoWWW/.pic={
    \draw (0,0.2)--(0,0)--(0.3,0)--(0.3,0.2);
  }
}

\tikzset{
  nconeoneinsidetwoWW/.pic={
	\path (0,0) pic {ncone};
    \path (0.1,0) pic {nctwoWW}; 
	\path (0.2,0.1) pic {ncone};
  }
}

\usepackage{forest}
\forestset{
  decor/.style = {
    label/.expanded = {[inner sep = 0.2ex, font=\unexpanded{\tiny}]right:{$#1$}}
  },
  root/.style = {minimum size = 0.1ex},
  decorated/.style = {
    for tree = {
      circle, fill, inner sep = 0.3ex, minimum size = 1.ex,
      grow' = south, l = 0, l sep = 1.2ex, s sep = 0.7em,
      fit = tight, parent anchor = center, child anchor = center,
      delay = {decor/.option = content, content =}
    }
  },
  default preamble = {decorated, root},
  begin draw/.code={\begin{tikzpicture}[baseline={([yshift=-0.5ex]current bounding box.center)}]},
}

\begin{document}

\newcommand{\arXivNumber}{2204.01445}

\renewcommand{\thefootnote}{}

\renewcommand{\PaperNumber}{038}

\FirstPageHeading

\ShortArticleName{Shifted Substitution in Non-Commutative Multivariate Power Series}

\ArticleName{Shifted Substitution in Non-Commutative\\ Multivariate Power Series with a View Toward\\ Free Probability\footnote{This paper is a~contribution to the Special Issue on Non-Commutative Algebra, Probability and Analysis in Action. The~full collection is available at \href{https://www.emis.de/journals/SIGMA/non-commutative-probability.html}{https://www.emis.de/journals/SIGMA/non-commutative-probability.html}}}

\Author{Kurusch~EBRAHIMI-FARD~$^{\rm a}$, Fr\'ed\'eric~PATRAS~$^{\rm b}$, Nikolas~TAPIA~$^{\rm c}$\newline and Lorenzo~ZAMBOTTI~$^{\rm d}$}

\AuthorNameForHeading{K.~Ebrahimi-Fard, F.~Patras, N.~Tapia and L.~Zambotti}

\Address{$^{\rm a)}$~Department of Mathematical Sciences, Norwegian University of Science and Technology,\\
\hphantom{$^{\rm a)}$}~NO 7491 Trondheim, Norway}
\EmailD{\href{mailto:kurusch.ebrahimi-fard@ntnu.no}{kurusch.ebrahimi-fard@ntnu.no}}
\URLaddressD{\url{https://folk.ntnu.no/kurusche/}}

\Address{$^{\rm b)}$~Universit\'e C\^ote d'Azur, CNRS, UMR 7351, Parc Valrose, 06108 Nice Cedex 02, France}
\EmailD{\href{mailto:patras@unice.fr}{patras@unice.fr}}
\URLaddressD{\url{https://math.unice.fr/~patras/}}

\Address{$^{\rm c)}$~Weierstra{\ss}-Institut Berlin and Technische Universit\"at Berlin, Berlin, Germany}
\EmailD{\href{mailto:tapia@wias-berlin.de}{tapia@wias-berlin.de}}
\URLaddressD{\url{http://wias-berlin.de/people/tapia/}}

\Address{$^{\rm d)}$~LPSM, Sorbonne Universit\'e, CNRS, Universit\'e Paris Cit\'e,\\
\hphantom{$^{\rm d)}$}~4 Place Jussieu, 75005 Paris, France}
\EmailD{\href{mailto:zambotti@lpsm.paris}{zambotti@lpsm.paris}}
\URLaddressD{\url{https://www.lpsm.paris/users/zambotti/}}

\ArticleDates{Received April 05, 2022, in final form May 29, 2023; Published online June 08, 2023}

\Abstract{We study a particular group law on formal power series in non-commuting variables induced by their interpretation as linear forms on a suitable graded connected word Hopf algebra. This group law is left-linear and is therefore associated to a pre-Lie structure on formal power series. We study these structures and show how they can be used to recast in a group theoretic form various identities and transformations on formal power series that have been central in the context of non-commutative probability theory, in particular in Voiculescu's theory of free probability.}


\Classification{16T05; 16T10; 16T30; 17A30; 46L53; 46L54}

\renewcommand{\thefootnote}{\arabic{footnote}}
\setcounter{footnote}{0}

\section{Introduction}
\label{sec:intro}

This work aims at explicitly relating two approaches to key arguments in Voiculescu's theory of free probability and related areas \cite{mingospeicher_17,speichernica,voiculescu_95}. On the one hand, the common approach by formal power series, on the other hand, a more recent one that relies on group-theoretical and Hopf algebraic arguments. For that purpose, we introduce and study various group, Lie and pre-Lie structures on formal power series in non-commuting indeterminates. We obtain as a by-product a dictionary between the two approaches that allows to translate various Hopf algebraic constructions into (non-trivial and non-standard) operations on formal power series.

Let us be more precise. In recent work (see~\cite{EFPa2019} and references therein), a shuffle group theoretic approach to moment-cumulant and cumulant-cumulant relations in non-commutative probability was proposed. In this setting, various families of cumulants, namely monotone, (conditionally) free, and Boolean, are understood as elements in the Lie algebra $\mathfrak{g}$ of infinitesimal characters over a particular combinatorial word Hopf algebra $H$. The series of moments is identified in turn with a particular element in the group $G$ of Hopf algebra characters on $H$. Three exponential-type maps happen to relate bijectively the group $G$ and its Lie algebra $\mathfrak{g}$; they therefore imply relations between the aforementioned families of cumulants and moments as well as relations amid the different types of cumulants.

On the other hand, it is well-known that the relations between moments and cumulants as well as the relations between the different families of cumulants can be concisely described in terms of multivariate generating functions, i.e., non-commuting formal power series \cite{AHLV2015,speichernica}. It is therefore natural to look for a precise understanding of the connection between formal power series in non-commuting variables with scalar-valued coefficients and the properties of linear forms on the aforementioned Hopf algebra $H$.

The shuffle algebra approach was further developed in \cite{EFPTZ_18, EFPTZ2021} with respect to Wick polynomials. In the later reference, it was shown that free, Boolean, and conditionally free Wick polynomials, introduced and studied in great detail by Anshelevich in a series of papers \cite{Anshelevich_04,Anshelevich_09a,Anshelevich_09b}, can be defined and related through the action of the group $G$ on the identity map in the space of endomorphisms on the tensor algebra $T(\mathcal{A})$ defined over the underlying non-commutative probability space $(\mathcal{A},\varphi)$. The paper \cite{EFPTZ2021} extended to non-commutative probability the setting of the previous work on Wick polynomials by the
same authors \cite{EFPTZ_18}.

In this work, we identify and study a particular group law (resp.~pre-Lie and Lie products as well as other operations and structures) on formal power series in non-commuting variables. It is induced, as we just alluded to, by the interpretation of the latter as linear forms on the Hopf algebra $H$.
This new group law is left-linear\footnote{See Definition~\ref{def:left-lin} below.}. It generalizes to the multivariate case --up to an isomorphism-- the group of tangent-to-identity formal power series that gives rise to the classical Fa\`a di Bruno Hopf algebra. Being left-linear, the group law is therefore associated to a novel pre-Lie structure defined on formal power series. Eventually, we note that our approach resembles in various respects the link between Butcher's group of $B$-series in numerical analysis~\cite{Butcher21} and a Hopf algebra of non-planar rooted trees described by Connes and Kreimer \cite{Brouder00,CEFM2011,CK98}. See Remark~\ref{butcher} below. We study these phenomena and show how they can be used to recast in a group-theoretic form various identities and transformations on formal power series that have shown to be central in the context of non-commutative probability theory, in particular in Voiculescu's theory of free probability \cite{speichernica}.

The paper is organised as follows. In Section~\ref{sec:ncpowerseries}, we describe the new group law on multivariate formal power series. Section~\ref{sec:pre-Lie} is dedicated to the corresponding pre-Lie and Lie structures. Sections~\ref{sec:Hopf} and~\ref{sec:shuffle} make the connections with the Hopf and shuffle algebraic viewpoints and explain how key operations on linear forms on the Hopf algebra $H$ transport to formal power series. The last section describes explicit connections with free, Boolean and monotone probability.

\section{The shifted composition group law}\label{sec:ncpowerseries}

The set of (strictly) positive integers is denoted by $\N$ and all structures are considered over the ground field $\mathbb K$ (of characteristic zero). Let $x=\{x_1,x_2,x_3,\dotsc\}$ be a set of formal non-commutative variables and let $A$ be a commutative $\mathbb K$-algebra with unit $1_A$ (or simply $1$ when no confusion can arise). The set of non-empty finite sequences of positive integers is denoted~$\N^*$. We will use the common word notation for elements in $\N^*$. The empty word is denoted $\mathbf{1}$, by convention it does not belong to $\N^*$. The subset $\N^k \subset \N^*$ contains words $w=i_1 \cdots i_k$ of length $|w|=k$, that is, sequences with exactly $k$ letters.

We consider the ring
\begin{equation*}
	R := A \langle\langle x_1,x_2,x_3,\dotsc\rangle\rangle
\end{equation*}
of non-commutative formal power series in these variables with coefficients in $A$. A typical element $f=f(x) \in R$ has the form
\[
	f(x) = f_0 + \sum_{k=1}^\infty \sum_{(i_1,\dotsc,i_k)\in \N^k}f_{i_1\dotsm i_k} x_{i_1} \dotsm x_{i_k}
\]
with coefficients $f_0, f_{i_1\dotsm i_k} \in A$. It is convenient to use word notation for such a series
\begin{equation*}
	f(x) = \sum_{w \in \N^* \cup \{\mathbf{1}\}} f_{w} x_{w} .
\end{equation*}
Here we have associated to a word $w=i_1 \dotsm i_m \in {\N}^m$ the non-commutative monomial $x_w := x_{i_1}\dotsm x_{i_m}$, with the convention that $x_\mathbf{1}=1_A$ and $f_\mathbf{1}=f_0$.
Multiplication in the ring $R$ is known as Cauchy product, which is denoted for elements $f,g \in R$ by
\begin{equation}
\label{prodGp}
	fg(x) := \sum_{w \in \N^* \cup \{\mathbf{1}\}} (fg)_w x_w.
\end{equation}
The coefficient of the empty word is $f_0g_0$ and $(fg)_w \in A$, for $w \in \N^*$ is defined to be
\[
	(fg)_{i_1\dotsm i_m}
	:= f_{i_1\dotsm i_m}g_0 + f_0g_{i_1\dotsm i_m}
	+\sum_{j=1}^{m-1} f_{i_1 \dotsm i_j} g_{i_{j+1} \dotsm i_m}.
\]
This can be compactly formulated in terms of the deconcatenation coproduct on words in $\N^*$
\begin{equation*}
	(fg)_{i_1\dotsm i_k} = m_A\big(\hat{f} \otimes \hat{g}\big) \delta(i_1\dotsm i_k),
\end{equation*}
where $\delta(\mathbf{1}) := \mathbf{1} \otimes \mathbf{1}$ and
\begin{equation*}
	\delta(i_1\dotsm i_k) = i_1\dotsm i_k \otimes \mathbf{1} +
						\mathbf{1} \otimes i_1\dotsm i_k
						+\sum_{j=1}^{k-1} {i_1\dotsm i_j} \otimes {i_{j+1} \dotsm i_k}.
\end{equation*}
The maps $\hat{f}$ and $\hat{g}$ are defined to be linear on the linear span of elements in $\N^*$ with values in~$A$, i.e., $\hat{f}(w):=f_w$, $\hat{g}(w):=g_w$, and extended to the empty word, $\hat{f}(\mathbf{1})=f_0$ and $\hat{g}(\mathbf{1})=g_0$.

In the following, we consider two distinct subsets of the ring $R$, which will be denoted $G^1$ and $G^0$. The former consists of elements in $R$ with unit constant coefficient
\begin{equation*}
	G^1:=\{f(x) \in R\ |\ f_0=1_A\}.
\end{equation*}
Elements $h \in G^1$ are written $h=1_A + h^{\prime}$. One verifies that $G^1$ forms a group under the usual multiplication \eqref{prodGp}. For $f \in G^1$, the coefficients of its inverse $f^{-1} \in G^1$ can be explicitly computed starting from those of $f(x)$. On the other hand, the set $G^0$ contains elements in $R$ with zero constant coefficient
\begin{equation*}
	G^0:=\{f(x) \in R\ |\ f_0=0\}.
\end{equation*}
We introduce also the set of so-called ``tangent-to-identity'' elements (according to the terminology of dynamical systems)
\begin{equation*}
	G^c:=\{f(x) \in R\ |\ f_0=0,\, f_i=1_A,\, i \in \N \}.
\end{equation*}

On the set $G^1$, we consider a new product by combining composition and the Cauchy product~\eqref{prodGp}. Before giving the definition, we introduce some notation. For $g \in G^1$, a new set of transformed variables
\begin{equation*}
	xg(x):=\{{(xg(x))}_1,{(xg(x))}_2,{(xg(x))}_3,\dotsc\}
\end{equation*}
is defined for $ i \in \N$ by
\begin{align*}
	(xg(x))_i := x_ig(x)
	= \sum_{w \in\N^* \cup \{ \mathbf{1} \} } g_{w} x_{iw}
	= x_i + \sum_{w \in\N^*} g_{w} x_ix_w.
\end{align*}
For words $w=i_1 \cdots i_l \in \N^*$, we set
\begin{equation*}
	(xg(x))_{i_1 \cdots i_l} := x_{i_1}g(x) \cdots x_{i_l}g(x).
\end{equation*}
\begin{Definition}
For $f,g \in G^1$, the {\it shifted composition} group law is defined:
\begin{equation}
\label{monoprod}
	 (f \bullet g)(x):=q g(x)f(xg(x)).
\end{equation}
\end{Definition}

\begin{Remark}
Later, it will become clear that from the shuffle algebra viewpoint on non-commutative probability, shifted composition corresponds to {\it monotone composition}.
\end{Remark}

Observe that, since $g \in G^1$, we have $x_ig \in G^c \subset G^0$ for all $i>0$. The composition $f(xg(x))=1_A + f'(xg(x))$ is therefore again an element in $G^1$. In the one-dimensional case (corresponding to a single variable $x=x_1$), the set $G^c$, whose elements are then called tangent-to-identity formal diffeomorphisms (as they are formal diffeomorphisms of the one-dimensional line), is equipped with a group law by the composition of univariate formal power series. One has furthermore the linear isomorphism $G^1 \cong G^c$ given by $\mu\colon f \longmapsto xf$. We therefore get
\begin{equation*}
	\mu(f \bullet g)(x)
	=xg(x)f(xg(x))
	=\mu(f)(\mu(g)(x)).
\end{equation*}

\begin{Corollary}
In the one-dimensional case, $G^1$ is isomorphic to the group of tangent-to-identity formal diffeomorphisms.
\end{Corollary}

We turn now back to the general case and explicitly compute the product for $f=1_A+f'$ and $g=1_A+g'$ in $G^1$
\allowdisplaybreaks
\begin{gather}
 (f\bullet g)(x) =g(x)f(xg(x)) 		
 =1_A + g^{\prime}(x) + f^{\prime}(xg(x))
 			+ \sum_{u,v \in \N^*}g_uf_v x_{u} (xg(x))_{v}\label{newproduct} \\
 \hphantom{(f\bullet g)(x)}{}
 = 1_A + g^{\prime}(x) + f^{\prime}(xg(x))
 	 + \sum_{\substack{u \in \N^*,\,v=i_1\dotsm i_k \in \N^*\\ u_1,\dotsc,u_k \in \N^*\cup \{ \mathbf{1} \}}}
 f_vg_ug_{u_1}\dotsm g_{u_k} x_{u}x_{i_1}x_{u_1}\dotsm x_{i_k}x_{u_k}.\nonumber
 \end{gather}
The new product on $G^1$ is associative. Indeed, an explicit computation shows that
\begin{align*}
 (f \bullet g) \bullet h(x) = h(x)(f \bullet g)(xh(x))= h(x)g(xh(x)) f(x h(x)g(xh(x))).
\end{align*}
We note that $h(x)g(xh(x))=g \bullet h (x)$ and
\begin{equation*}
		f(x h(x)g(xh(x)))=1_A+\sum_{u \in \N^*}f_u (x h(x)g(xh(x)))_u,
\end{equation*}
where
\begin{equation*}
	(x h(x)g(xh(x)))_i
	=(xh(x))_ig(xh(x))
	=x_ih(x)g(xh(x)).
\end{equation*}
Compare this with
\begin{align*}
	f \bullet (g \bullet h)(x) &= (g \bullet h)(x) f(x({g \bullet h})(x))= h(x)g(xh(x)) f(x(g \bullet h)(x)).
\end{align*}
Here
\begin{equation*}
	f(x({g \bullet h})(x))=1_A + \sum_u f_u (x(g \bullet h)(x))_u
\end{equation*}
and
\begin{equation*}
	(x(g \bullet h)(x))_i
	=x_i (g \bullet h)(x)
	=x_ih(x)g(xh(x)),
\end{equation*}
which shows associativity of the shifted composition on $G^1$, that is,
\begin{equation*}
	(f \bullet g) \bullet h(x) = f \bullet (g \bullet h)(x).
\end{equation*}
The unit for shifted composition \eqref{monoprod} is $1_A$. Indeed
\begin{equation*}
	1_A \bullet g(x)=g(x)1_A=g(x)=g \bullet 1_A=1_Ag(x).
\end{equation*}
In fact, we have

\begin{Proposition}\label{thm:monogroup}
$G^\bullet:=\big(G^1,\bullet\big)$ is a non-commutative group with unit~$1_A$.
\end{Proposition}

\begin{proof}
The identity
\begin{equation*}
	1_A=f^{\bullet -1} \bullet f(x)=f(x)f^{\bullet -1} (xf(x))
\end{equation*}	
allows to recursively (and uniquely) compute the coefficients of the series $f^{\bullet -1}$ from those of~$f$.\looseness=1
\end{proof}

\begin{Remark}\label{rmk:known}
We note that the definition of the product \eqref{monoprod} is motivated by the shuffle convolution product defined in \cite{ebrahimipatras_15}, as will become clear in Section \ref{sec:Hopf}. Moreover, we observe that a variation of the sequence defining the product \eqref{monoprod} appears in Anshelevich \cite[Corollary~7]{anshelevich2010}. In fact, here the sequence defining the product $f \bullet g$ appears in reversed order, i.e., $f(zg(z))g(z)=g(z)^{-1} (f \bullet g)(z)g(z)$, in relation to a statement on positive definiteness for the coefficient sequences. However, its group-theoretical properties have not been considered by Anshelevich. Also, looking at Anshelevich's free Wick polynomials \cite{Anshelevich_04} from a shuffle Hopf algebraic perspective, as was done in our work \cite{EFPTZ2021}, one may extract the product \eqref{monoprod} at the level of formal power series. See reference \cite[Theorem~3.10, Proposition~3.12, equation~(3.48)]{Anshelevich_04}.
\end{Remark}

\section{The pre-Lie and Lie algebraic structures}\label{sec:pre-Lie}

We recall the notion of left-linear group and some of its properties -- for details, we refer the reader to the recent book \cite[Section~6.4]{CP}. Consider local coordinates $\mathbf x = \big(x^1,\dots,x^n\big)$ on a~Lie group~$G$ in the neighborhood of the identity element $e$, with the property $x^i(e)= 0$, for $1 \le i \le n$. For notational convenience, we identify the system of local coordinates with the element of the group.
Using these coordinates, we assume that the group law reads
\begin{equation*}
	\mathbf z
	= F(\mathbf x;\mathbf y)
	= \sum_{\substack{p\geq0\\q\geq0}} F_{p,q}(\mathbf x;\mathbf y),
\end{equation*}
if $\mathbf z = \mathbf x \cdot \mathbf y$ and where $F_{p,q} \big(x^1,\dots,x^n;y^1,\dots,y^n\big)$ is a polynomial
in $2n$ variables, homogeneous of degrees $p$ and $q$ in the variables~$\mathbf x$ respectively~$\mathbf y$. Then the difference $F_{1,1}(\mathbf x;\mathbf y) - F_{1,1}(\mathbf y;\mathbf x)$ defines the Lie bracket in the Lie algebra $\mathfrak g$ of $G$.

\begin{Definition}
\label{def:left-lin}
The group $G$ is said to be left-linear if $F_{p,q}=0$ for $p\geq 2$, that is, if $F(\mathbf x;\mathbf y)-\mathbf y$ is linear in $\mathbf x$.
\end{Definition}

Let us write $\mathbf x \triangleleft \mathbf y$ for $F_{1,1} (\mathbf x;\mathbf y)$. Then it holds in general that the tangent space $\mathfrak g$ to a~left-linear Lie group is equipped with the binary operation $\triangleleft$ with the structure of a (right) pre-Lie algebra. In fact, from the latter, the Lie algebra structure is inherited as $\mathbf x\triangleleft \mathbf y-\mathbf y\triangleleft \mathbf x$. That is, for arbitrary $\mathbf x$, $\mathbf y$, $\mathbf z$, we have the (right) pre-Lie identity
\begin{equation*}
	(\mathbf x\triangleleft \mathbf y)\triangleleft \mathbf z- \mathbf x\triangleleft
	(\mathbf y\triangleleft\mathbf z)
	= (\mathbf x\triangleleft \mathbf z)\triangleleft \mathbf y- \mathbf x\triangleleft
	(\mathbf z\triangleleft\mathbf y).
\end{equation*}

The definition extends to the infinite-dimensional case -- keeping the requirement that the components of $F_{p,q}$ be polynomials in the coordinates.
In particular, the group $\big(G^1,\bullet\big)$ is an (infinite-dimensional) left-linear group.
Indeed, from equation~\eqref{newproduct} we see that
\begin{equation}
\label{eq:leftlinear}
\begin{split}
 (f\bullet g-g)(x) &= g(x)f'(xg(x)) \\
&	= f^{\prime}(xg(x))
+ \smashoperator[r]{\sum_{\substack{u \in \N^*,\, v=i_1\dotsm i_k \in \N^*\\ u_1,\dotsc,u_k \in \N^*\cup \{ \mathbf{1} \}}}}
 			f_vg_ug_{u_1}\dotsm g_{u_k} x_{u}x_{i_1}x_{u_1}\dotsm x_{i_k}x_{u_k},
\end{split}
\end{equation}
and this expression is linear in the coordinates $(f_v)_{v\in\N^\ast}$ of $f$.

\begin{Proposition}
The tangent space $\mathfrak g=G^0$ at $1$ to the left-linear group $\big(G^1,\bullet\big)$ is a right pre-Lie algebra.
\end{Proposition}

\begin{proof}Let us check the property explicitly.
	We consider the coordinates in the basis of words $x_{i_1}\cdots x_{i_n}$. Using the notation of the previous section and equation~\eqref{eq:leftlinear}, we obtain for $F_{1,1}$ and~$\triangleleft$:
\begin{equation}
\label{eq:pre-Lie}
	x_{i_1}\cdots x_{i_n}\triangleleft x_{j_1}\cdots x_{j_m}
	=\sum\limits_{k=0}^n x_{i_1}\cdots x_{i_k} x_{j_1}\cdots x_{j_m}x_{i_{k+1}}\cdots x_{i_n}.
\end{equation}
Denoting the insertion of $\mathbf y:= x_{j_1}\cdots x_{j_m}$ in position $k$ inside $\mathbf x:=x_{i_1}\cdots x_{i_n}$ by
\begin{equation*}
	x_{i_1}\cdots x_{i_k} \mathbf y \, x_{i_{k+1}}\cdots x_{i_n}
	:=x_{i_1}\cdots x_{i_k} x_{j_1}\cdots x_{j_m}x_{i_{k+1}}\cdots x_{i_n},
\end{equation*}
we get, with a self-explaining notation,
\begin{gather*}
	(\mathbf x\triangleleft \mathbf y)\triangleleft \mathbf z - \mathbf x\triangleleft
	(\mathbf y\triangleleft\mathbf z)=
\sum\limits_{0 \le k < l \le n} x_{i_1}\cdots x_{i_k}
	\mathbf y \, x_{i_{k+1}}\cdots x_{i_l}\mathbf z \, x_{i_{l+1}}\cdots x_{i_n}+
\\
\hphantom{(\mathbf x\triangleleft \mathbf y)\triangleleft \mathbf z - \mathbf x\triangleleft
	(\mathbf y\triangleleft\mathbf z)=}{}
+ \sum\limits_{0 \le k < l \le n}
	x_{i_1}\cdots x_{i_k} \mathbf z \, x_{i_{k+1}}\cdots x_{i_l}\mathbf y \, x_{i_{l+1}}\cdots x_{i_n}.
\end{gather*}
As this expression is symmetric in $\mathbf y$ and $\mathbf z$, we deduce that the product $\triangleleft$ is (right) pre-Lie with associated Lie bracket
\begin{align*}
 [x_{i_1}\cdots x_{i_n} ,x_{j_1}\cdots x_{j_m}] ={}& x_{i_1}\cdots x_{i_n} \triangleleft x_{j_1}\cdots x_{j_m}
		-x_{j_1}\cdots x_{j_m} \triangleleft x_{i_1}\cdots x_{i_n} \\
={}&\sum\limits_{k=1}^{n-1}x_{i_1}\cdots x_{i_k} x_{j_1}\cdots x_{j_m}x_{i_{k+1}}\cdots x_{i_n}\\
 &{} - \sum\limits_{l=1}^{m-1}x_{j_1}\cdots x_{j_l} x_{i_1}\cdots x_{i_n}x_{j_{l+1}}\cdots x_{j_m}.
 \tag*{\qed}
\end{align*}
\renewcommand{\qed}{}\end{proof}

\begin{Remark}
In the single variable case we deduce from \eqref{eq:pre-Lie} that
\begin{equation}
	\label{eq:prelie.x}
	x^n\triangleleft x^m=(n+1)x^{n+m},
\end{equation}
so that \(\big[x^n,x^m\big]=(n-m)x^{n+m}\).
The corresponding pre-Lie algebra is isomorphic to the pre-Lie algebra associated to the group of tangent-to-identity formal diffeomorphisms of the line. The Lie algebra is, up to isomorphism, the Lie algebra of primitive elements in the cocommutative Hopf algebra dual of the Fa\`a di Bruno Hopf algebra.
\end{Remark}

\section{Coordinate Hopf algebra}\label{sec:Hopf}

From now on, we will use freely general and standard results and notions from the theory of bialgebras and Hopf algebras such as convolution products, characters, infinitesimal characters, and the Baker--Campbell--Hausdorff formula. The reader is referred to \cite{CP} for details.

The group $\big(G^1,\bullet\big)$ is pro-unipotent (that is, an inverse limit of unipotent groups). This can be deduced for example from the observation that the ring $R$ of formal power series is the inverse limit of the quotients $A<x_1,\ldots, x_k,\ldots >/I(n)$, where $I$ is the ideal of the algebra of non-commutative polynomials spanned by degree $n$ monomials $x_{i_1} \cdots x_{i_n}$. As such, $\big(G^1,\bullet\big)$ is the group of characters of a commutative Hopf algebra (see \cite[Section 3.6]{CP}). Technically, this Hopf algebra is, as an algebra, the direct limit of the polynomial algebras over finite subsets of the set of coordinate functions $(f_v)_{v\in\N^\ast}$ on $G^1$.
The algebra structure is the product of polynomials. The coproduct is obtained automatically by dualizing the group law.

However, it is convenient to identify $\big(G^1,\bullet\big)$ with the group of characters of a larger and, more importantly, non-commutative Hopf algebra. This will put at our disposal the tools and techniques available for studying shuffle groups in the sense of \cite{EFPa2019}.

Recall that $\N^*$ is the free semigroup over the alphabet of positive integers, $\N=\{1,2,3,\ldots\}$. Let $V$ denote the vector space spanned by it. Elements in $V$ are linear combinations of non-empty words in the letters of the alphabet and it naturally possesses the structure of a non-unital associative algebra, the product being the unique bilinear extension of the concatenation of words. We write $V^+$ for the augmentation of the algebra $V$ by a unit (that we identify as before with the empty word denoted here $\mathbf{1}$).

There is a natural bijection $\Lambda \colon \mathrm{Lin}\big(V^+,A\big) \to R$ given by
\begin{equation}
\label{bij}
	\Lambda(\phi):= \phi(\mathbf{1})+\sum_{w \in \N^*} \phi(w)x_w,
\end{equation}
where $\Lambda(\phi)$ can be understood as a generating series for the functional $\phi$.
Let
\[
	T(V) :=\bigoplus_{n\ge0}V^{\otimes n}
\]
be the tensor algebra over $V$, where $V^{\otimes 0}\cong \mathbb{K}\1$ is one-dimensional. To avoid confusion over the use of several tensor products, we denote elements $w_1\otimes\cdots\otimes w_k$ of $T(V)$, where $w_i\in V$, by $w_1| \cdots | w_k$, that is, by inserting vertical bars instead of the usual tensor product symbol.\footnote{This resembles the algebraic part of the double bar construction \cite{Baues81}.} In particular, the concatenation product $m_{|}\colon T(V) \otimes T(V) \to T(V)$ sends words $w_1,w_2 \in V$ to $m_{|}(w_1 \otimes w_2):=w_1 | w_2$ and, more generally, $m_|((w_1|\cdots |w_k)\otimes (w_{k+1}|\cdots |w_n))=w_1|\cdots |w_k|w_{k+1}|\dots |w_n$. We denote $T_+(V) :=\bigoplus_{n\ge1}V^{\otimes n}$ the augmentation ideal. The bijection above from $\mathrm{Lin}(V^+,A)$ to $R$ extends to a linear map from $\mathrm{Lin}(T(V),A)$ to $R$, which we still write abusively $\Lambda$, using the formula \eqref{bij}. It is important to notice that the value of $\phi$ on the spaces $V^{\otimes n}$ are not taken into account for $n\geq 2$.

Let us denote by $\mathcal{G}(A) := \mathrm{Hom}_{\scriptscriptstyle \mathrm{alg}}(T(V),A)$ the set of algebra morphisms, i.e., multiplicative unital maps (or characters) in $\mathrm{Lin}(T(V),A)$, and by $\mathcal{L}(A)$ the set of so-called infinitesimal characters in $\mathrm{Lin}(T(V),A)$. These are the linear maps that vanish on $\1$ as well as non-trivial products of words, i.e., on $\bigoplus_{n\ge2}V^{\otimes n}$. By their very definition, elements in $\mathcal{G}(A) $ and $\mathcal{L}(A)$ are entirely characterized by their values on the elements of the semigroup $\N^*$ that form a basis of~$V$. By restricting $\Lambda$ to $\mathcal{G}(A)$, respectively~$\mathcal{L}(A)$, the existence of two bijections of sets follows:%
\begin{equation}\label{grell}
	\Lambda_{\mathrm{gr}}\colon \ \mathcal{G}(A)\to G^1,
	\qquad
	\Lambda_{\mathrm{Lie}}\colon \ \mathcal{L}(A)\to G^0.
\end{equation}

Given a word $w=a_1\dotsm a_n \in V$ and a subset $S=\{i_1<\dotsb<i_k\} \subseteq [n]$ we set $w_S:=a_{i_1} \dotsm a_{i_k} \in V$. The complement $S^c:= [n] \setminus S$ can be written as the disjoint union of $m=m(S)$ maximal intervals $J_1^S,\dotsc, J_m^S$ defined through the set $S$.

We introduce a coproduct $\Delta \colon V \to V\otimes T(V)$ by setting $\Delta\1=\1\otimes\1$ and for $w=a_1\dotsm a_n \in V$%
\begin{equation}
\label{delteq}
	\Delta(a_1\dotsm a_n):=\sum_{S\subseteq[n]}
	w_S \otimes w_{J^S_1}\vert\dotsm\vert w_{J^S_m},
\end{equation}
which is multiplicatively extended to $T(V)$:
\begin{equation*}
	\Delta(w_1|\cdots|w_n):=\Delta(w_1)\cdots\Delta(w_n)\in T(V)\otimes T(V).
\end{equation*}

\begin{Theorem}[\cite{ebrahimipatras_15}]
The space $T(V)$ with product $m_{|}$ and coproduct $\Delta$ defined in \eqref{delteq} is a graded connected non-commutative non-cocommutative bialgebra, that we denote by $H:=(T(V),\Delta,m_{|},\epsilon,\eta)$.
\end{Theorem}

Here, the unit map $\eta\colon \mathbb K\to T(V)$, respectively~the counit map $\epsilon\colon T(V)\to \mathbb K$, are the obvious inclusion of, respectively~projection to the scalar component $\mathbb K=V^{\otimes 0}$ in the space $T(V)$.

Let now $\mathrm{Lin}(T(V),A)$ denote the space of linear maps taking values in the unital commutative algebra $A$. Recall that this space has a natural unital algebra structure given by convolution, that is, for $\phi, \psi \in\mathrm{Lin}(T(V),A)$ we set
\[
	\phi * \psi := m_A(\phi \otimes \psi) \Delta,
\]
where $m_A$ denotes the product in $A$. The unit for the convolution product is given by $\varepsilon_A:=\eta_A \circ \epsilon$, where $\epsilon$ is the counit of $T(V)$ and $\eta_A\colon \mathbb K\to A$ is the unit-map of $A$ ($\eta_A(1):=1_A$).

Recall also that $H$ is automatically a Hopf algebra. It is well known that the set $\mathcal{G}(A)$ forms a group under convolution; the inverse of an element is given by composition with the antipode of $T(V)$. Similarly, $\mathcal{L}(A)$ is a Lie algebra for the Lie bracket obtained by anti-symmetrizing the convolution product, i.e., $[\phi,\psi]=\phi*\psi - \psi*\phi$. We also have the existence of inverse bijections
\begin{equation*}
	\exp^\ast\colon \ \mathcal{L}(A)\to \mathcal{G}(A), \qquad \log^\ast\colon \ \mathcal{G}(A)\to \mathcal{L}(A),
\end{equation*}
with $\exp^\ast\circ\log^\ast={\rm id}_{\mathcal{G}(A)}$ and $\log^\ast\circ\exp^\ast={\rm id}_{\mathcal{L}(A)}$.

\begin{Remark}
\label{remmono}
In the context of non-commutative probability, it has been shown elsewhere that if the linear unital map $\varphi\colon \mathcal{A} \to \mathbb{K}$ on a non-commutative probability space $(\mathcal{A},\varphi)$ is extended to a character $\Phi$ on the double tensor algebra over $\mathcal{A}$, suitably equipped with a Hopf algebra structure very similar to the one we defined on $T(V)$, then $\log^\ast(\Phi)$ computes the associated multivariate monotone cumulants. We refer to \cite{ebrahimipatras_17} for details. The reader should keep in mind that these results are in the background of the developments in the present article.
\end{Remark}

\begin{Theorem}
\label{theorem:monotone}
The map $\Lambda_{\mathrm{gr}}$ defines a group isomorphism between the group $(\mathcal{G}(A),\ast)$ and the group $\big(G^1,\bullet\big)$.
\end{Theorem}

\begin{proof}
We already know that the map is a bijection. We would like to show that for characters $\phi,\psi \in \mathcal{G}(A)$
\begin{equation*}
	\Lambda_{\mathrm{gr}}(\phi * \psi )(x)= f \bullet g (x),
\end{equation*}
where $f (x) := \Lambda_{\mathrm{gr}}(\phi)(x)$ and $g(x) := \Lambda_{\mathrm{gr}}(\psi)(x)$.

We first recall that for a word $v=i_1\dotsm i_k$ and element $g \in G^1$, we have
\begin{gather*}
\begin{split}
 & (xg(x))_v
 =(xg(x))_{i_1} \dotsm (xg(x))_{i_k}
 =x_{i_1}g(x)x_{i_2}g(x) \cdots x_{i_k}g(x)\\
		&
\hphantom{(xg(x))_v}{}
=x_v + \sum_{\substack{u_1,\dotsc,u_k\in \{\1\} \cup \N^* \\ u_1 \cdots u_k \neq \1}}
 g_{u_1}\dotsm g_{u_k}x_{i_1}x_{u_1}\dotsm x_{i_k}x_{u_k}.
 \end{split}
\end{gather*}
Copying \eqref{newproduct}, we have
 \begin{align}
 f \bullet g(x)
 &=g(x)f(xg(x))
 =\sum_{u,v \in \{\1\} \cup \N^*}g_uf_vx_{u}(xg(x))_{v}\nonumber\\
 &=g(x) + \sum_{\substack{u,u_1,\dotsc,u_k\in \{\1\} \cup \N^* \\v=i_1\dotsm i_k \in \N^*}}
 f_vg_ug_{u_1}\dotsm g_{u_k}x_{u}x_{i_1}x_{u_1}\dotsm x_{i_k}x_{u_k}. \label{thesum1}
 \end{align}
Then, writing $f_w=\phi(w)$ and $g_u=\psi(u)$, the above sum collapses to
\begin{align*}
	g(x) &+ \sum_{\substack{u,u_1,\dotsc,u_k\in \{\1\} \cup \N^* \\v=i_1\dotsm i_k \in \N^*}}
	f_vg_ug_{u_1}\dotsm g_{u_k}x_{u}x_{i_1}x_{u_1}\dotsm x_{i_k}x_{u_k} \nonumber\\
 	&=1+ \sum_{w\in \N^*}(\phi*\psi)(w)x_w
 	 =\Lambda_{\mathrm{gr}}(\phi*\psi)(x). \nonumber
\end{align*}
The proof is complete.
\end{proof}

A similar calculation shows that the analog statement holds at the level of Lie algebras:

\begin{Theorem}\label{theorem:Liemonotone}
The map $\Lambda_{\mathrm{Lie}}$ defines a Lie algebra isomorphism between the Lie algebras $\mathcal{L}(A)$ and $\big(G^0,[- , - ]\big)$.
\end{Theorem}

\subsection{The BCH group law}\label{ssec:bch}

Recall now the Baker--Campbell--Hausdorff (BCH) formula in the free associative algebra over two variables $X$, $Y$:
\begin{equation*}
	\exp(X)\exp(Y)=\exp({\rm BCH}(X,Y)),
\end{equation*}
where ${\rm BCH}(X,Y)$ is an element in the free Lie algebra over $X$ and $Y$, that is, a linear combination of iterated Lie brackets of $X$ and $Y$ ($[X,Y]:=XY-YX$) such as $[X,Y]$, $[X,[X,Y]]$, $[[X,Y],[X,Y]]$, and so on. Setting $f\ast_{\scriptscriptstyle{\text{BCH}}}g:=\text{BCH}(f,g)$, this formula defines the BCH group law on the Lie algebra $\big(G^0,[- , - ]\big)$ of the infinite-dimensional group $\big(G^1,\bullet\big)$. A BCH group law is defined on $\mathcal{L}(A)$ similarly. Equivalently, it is defined by transportation of the group law on~$\mathcal{G}(A)$ along $\exp^\ast$: for $\phi, \rho$ in $\mathcal{L}(A)$, we have
\begin{equation*}
	{\rm BCH}(\phi,\rho):=\log^\ast(\exp^\ast(\phi)\ast\exp^\ast(\rho)).
\end{equation*}

\begin{Corollary}
The BCH group law on $\mathcal{L}(A)$ is transported by $\Lambda_{\mathrm{Lie}}$ to the BCH group law on~$G^0$.
\end{Corollary}

Notice that it follows from our arguments that there exists a bijection (in fact, an isomorphism) $\exp^G$ between $G^0$ and $G^1$ (recall that the former is the Lie algebra of the latter) given~by
\begin{equation*}
	\exp^G:=\Lambda_{\mathrm{gr}}\circ \exp^\ast\circ \Lambda_{\mathrm{Lie}}^{-1}
\end{equation*}
with inverse
\begin{equation*}
	\log^G:=\Lambda_{\mathrm{Lie}}\circ \log^\ast\circ \Lambda_{\mathrm{gr}}^{-1}.
\end{equation*}
These bijections are given by complex formulas (the same that relate monotone cumulants to moments in free probability, see our Remark~\ref{remmono} above).

\begin{Remark}\label{butcher}
An example of similar nature to the construction of the map $\Lambda$, resp.~$\Lambda_{\mathrm{gr}}$, $\Lambda_{\mathrm{Lie}}$, is provided by Butcher's group of $B$-series in numerical analysis \cite{Butcher21,HLW2006} and its link to a certain combinatorial Hopf algebra on rooted trees. We recall that a $B$-series may be characterised as the Taylor expansion of numerical integration schemes such as Runge--Kutta methods:
\begin{equation*}
	B(\alpha; hf,y) := \sum_{t \in \mathcal{T}} \alpha(t)F_{hf}[t](y),
\end{equation*}
where the sum on the righthand side runs over the set $\mathcal{T}$ of non-planar rooted trees, including the empty tree, and $\alpha$ is a function on $\mathcal{T}$ determined by the numerical method. The other objects involved are a smooth vector field $f$ on $\mathbb{R}^d$, the step size parameter $h \in \mathbb{R}$ and the map $F_f$ which associates a so-called elementary differential to a trees $t \in \mathcal{T}$ and the aforementioned vector field $f$ (it was first described by Cayley in the context of differential equations \cite{Cayley1889}). See \cite{HLW2006} for details. It turns out that composition of two $B$-series, i.e., $B(\alpha; hf,B(\beta; hf,y)) = B(\beta*\alpha; hf,y)$, is tightly linked to a combinatorial Hopf algebra defined on non-planar rooted trees. Indeed, the coefficients of the $B$-series $B(\beta*\alpha; hf,y)$ are computed in terms of the convolution product of the group of Hopf algebra characters over the Butcher--Connes--Kreimer Hopf algebra, \cite{Brouder00,CHV2010,CK98}.
\end{Remark}

\section{Half-shuffle products}\label{sec:shuffle}

The coproduct $\Delta$ on $V$, given in \eqref{delteq}, can be split into the sum of two so-called left and right half-coproducts
\begin{equation}
\label{deltaleft}
	\Delta_\prec(a_1\dotsm a_n)
	:= a_1\dotsm a_n \otimes \1 + \sum_{1\in S\subsetneq[n]}
	w_S\otimes w_{J^S_1}\vert\dotsm\vert w_{J^S_m}
\end{equation}
and
\begin{equation}
\label{deltaright}
	\Delta_\succ(a_1\dotsm a_n)
	:= \1 \otimes a_1\dotsm a_n + \sum_{\substack{1\not\in S\subsetneq[n] \\S \neq \varnothing}}
	w_S\otimes w_{J^S_1}\vert\dotsm\vert w_{J^S_m}.
\end{equation}
Both these half-coproducts are extended to $T_+(V)$ by defining
\begin{equation*}
	\Delta_\prec(w_1\vert w_2 \vert \dotsm \vert w_n)=\Delta_\prec(w_1)\Delta(w_2\vert \dotsm \vert w_n)
\end{equation*}
and similarly for $\Delta_\succ$, so that the coproduct \eqref{delteq} on $T(V)$ can be written as a sum, $\Delta=\Delta_\prec+\Delta_\succ$. It can be shown that the two half-coproducts define an \emph{unshuffle bialgebra} structure on~$T(V)$~\cite{ebrahimipatras_15}.

This induces a splitting of the convolution product on the dual side into a sum of a ``left half-shuffle product'' and a ``right half-shuffle product'' for $A$-valued linear forms on $T_+(V)$ (identified with $A$-valued linear forms on $T(V)$ that vanish on $\mathbb K$)
\[
	\phi \prec \psi := m_A(\phi \otimes \psi )\Delta_\prec,
	\qquad
	\phi \succ \psi := m_A(\phi \otimes \psi )\Delta_\succ
\]
such that the associative convolution product of such linear forms decomposes
\begin{equation}
\label{convolprod}
	\phi * \psi = \phi \succ \psi + \phi \prec \psi.
\end{equation}
The left and right half-shuffle products are then extended partially by setting
\[
	\phi\prec\varepsilon_A:=\phi,
	\qquad
	\varepsilon_A\succ\phi:=\phi,
	\qquad
	\phi\succ\varepsilon_A:=0,
	\qquad
	\varepsilon_A\prec\phi:=0.
\]
The products
$\varepsilon_A\prec\varepsilon_A$, $\varepsilon_A\succ\varepsilon_A$ are left undefined.

Associativity of the convolution product \eqref{convolprod} can be deduced from the fact that the space $(\mathrm{Lin}(T_+(V),A),{\prec},{\succ})$ is a (non-commutative) shuffle algebra \cite{ebrahimipatras_15} as the left half-shuffle product and a right half-shuffle product satisfy the shuffle identities:
\begin{gather*}
	(\phi \prec \psi) \prec \rho = \phi \prec (\psi * \rho), \\ 		
 	(\phi \succ \psi) \prec \rho = \phi \succ (\psi \prec \rho), \\		
 	 \phi \succ (\psi \succ \rho) = (\phi * \psi)\succ \rho. 		
\end{gather*}
Note that these are the identities satisfied by shuffle products in algebraic topology and products of iterated integrals of time-dependent matrices in classical calculus and stochastic integration \`a la Stratonovich \cite{EFP2022}.

\begin{Proposition}
\label{prop:gprod}
Let $\phi\in \mathrm{Lin}(T_+(V),A)$ and $\gamma\in {\mathcal G}(A)$. We set $f:=\Lambda(\phi) \in G^0$ and $g:=\Lambda_{\mathrm{gr}}(\gamma) \in G^1$.
Then we have
\begin{equation}\label{<1}
	\Lambda(\phi\prec\gamma)
	=\sum_{\substack{u_1,\dotsc,u_k\in \{\1\} \cup \N^* \\v=i_1\dotsm i_k\in \N^*}}
	f_vg_{u_1}\dotsm g_{u_k}x_{i_1}x_{u_1}\dotsm x_{i_k}x_{u_k}=f(xg(x)),
\end{equation}
respectively
\[
	\Lambda(\phi\succ\gamma)
	= \sum_{\substack{u_1,\dotsc,u_k \in \{\1\} \cup \N^* \\{\substack{u \in \N^* \\v=i_1\dotsm i_k \in \N^* }}}}
	 f_vg_ug_{u_1}\dotsm g_{u_k}x_{u}x_{i_1}x_{u_1}\dotsm x_{i_k}x_{u_k}
	 =(g(x)-1)f(xg(x)).
\]
\end{Proposition}

\begin{proof}
The statement follows by dualizing the formulas \eqref{deltaleft} and \eqref{deltaright}, using that $\gamma$ is a~character.
\end{proof}

Observe that when $\varepsilon_A+\phi \in{\mathcal G}(A)$, the decomposition of $\Lambda(\phi\ast\gamma)=\Lambda(\phi\prec\gamma+\phi\succ\gamma)$ reflects the splitting of the series in \eqref{thesum1} at the level of the sum over words $u$ as
 \begin{gather*}
g(x) + \sum_{\substack{u,u_1,\dotsc,u_k\in \{\1\} \cup \N^* \\v=i_1\dotsm i_k \in \N^*}}
 	f_vg_ug_{u_1}\dotsm g_{u_k}x_{u}x_{i_1}x_{u_1}\dotsm x_{i_k}x_{u_k}\\
	\qquad =1+ \sum_{\substack{u_1,\dotsc,u_k\in \{\1\} \cup \N^* \\v
	=i_1\dotsm i_k\in \N^*}}f_vg_{u_1}\dotsm g_{u_k}x_{i_1}x_{u_1}\dotsm x_{i_k}x_{u_k} \\
	\qquad \quad + \sum_{\substack{u_1,\dotsc,u_k \in \{\1\} \cup \N^* \\{\substack{u \in \N^* \\v=i_1\dotsm i_k \in \{\1\} \cup \N^* }}}}
	 f_vg_ug_{u_1}\dotsm g_{u_k}x_{u}x_{i_1}x_{u_1}\dotsm x_{i_k}x_{u_k}.
\end{gather*}
This splitting corresponds to the left and right half-shuffles in the non-commutative shuffle algebra $(\mathrm{Lin}(T(V),A),{\prec},{\succ})$.

We define now two binary operations mapping $G^1 \times G^0$ into $G^0$:
\[
	\big(g \curvearrowright f\big)(x):= f(xg(x))
	=\sum_{w\in \N^*} f_w(xg(x))_w \in G^0
\]
and
\[
	\big(g \curvearrowleft f\big)(x):= (g-1)f_{xg}(x)= (g(x)-1)f(xg(x)) =\sum_{\substack{w\in \N^* \\u \in \N^*}} f_wg_ux_u (xg(x))_w \in G^0.
\]
Let us consider two power series $f(x),g(x) \in G^1$, then the product \eqref{monoprod} can be written
\begin{equation*}
	f \bullet g=g(x)+\big(g \curvearrowright (f-1)\big)(x)+\big(g \curvearrowleft (f-1)\big)(x).
\end{equation*}
Hence, $\Lambda$ maps the group product $G^1$ to shifted composition whenever $\gamma$ is a character. On the other hand, the next statement may be considered as a linearized form the statement in Proposition~\ref{prop:gprod}.

\begin{Proposition}
\label{prop:lprod}
Let $\phi \in \mathrm{Lin}(T(V),A)$ and $\gamma \in {\mathcal L}(A)$. We set $f:=\Lambda(\phi) \in R$ and $g:=\Lambda_{\mathrm{Lie}}(\gamma) \in G^0$.
Then we have
\begin{equation*}
	\Lambda(\phi\prec\gamma)=\sum\limits_{j=1}^k\sum_{\substack{u\in \N^* \\v=i_1\dotsm i_k\in \N^*}}
	f_vg_ux_{i_1}\dotsm x_{i_j}x_ux_{i_{j+1}}\dotsm x_{i_k},
\end{equation*}
respectively
\begin{equation}\label{>1}
	\Lambda(\phi\succ\gamma)
	= \sum_{u,v \in \N^* } f_vg_ux_ux_v=g(x)f(x).
\end{equation}
In particular, going back to \eqref{eq:pre-Lie}, we find
\begin{equation*}
	\Lambda(\phi\prec\gamma +\phi\succ\gamma)=\Lambda(\phi\ast\gamma)=f\triangleleft g .
\end{equation*}
\end{Proposition}
Hence, the pre-Lie product defined in Section~\ref{sec:pre-Lie} can be seen as the linearization of the group law of $G^1$.

\begin{proof}
The statement follows again by dualizing the formulas \eqref{deltaleft} and \eqref{deltaright}, using that $\gamma$ is now an infinitesimal character.
\end{proof}

\begin{Remark}
There is a general difficulty with series: the space is \emph{too small} to build consistently all shuffle operations on it. This is why we always have to carefully distinguish what happens in the group $G^1$ and the Lie algebra $G^0$. Defining operations that would make sense simultaneously on the two and would also fit with what happens in $\mathrm{Lin}(T(V),A)$ is impossible. The conclusion is precisely that the shuffle approach provides an unifying and elegant algebraic alternative to the generating series framework.
\end{Remark}

\section{Link with non-commutative probability}\label{sec:NCproba}

Let us consider now $(B,\varphi)$ a non-commutative probability space over the complex numbers. That is, $B$ is an associative unital algebra over $A:=\mathbb C$ and $\varphi$ a $\mathbb C$-valued unital linear form on~$B$~\cite{speichernica}. Let $(b_n)_{n\in\N}$ be a countable family of non-commutative random variables in $B$ (that is, of elements of $B$).

In the setting of Section~\ref{sec:Hopf}, we associate to these data the linear form $\phi\colon V\to\mathbb C$ defined for words $w=i_1\cdots i_k\in\N^*$ by
\begin{equation*}
	\phi(w):=\varphi(b_{i_1}\cdot_{\scriptscriptstyle{B}} \cdots \cdot_{\scriptscriptstyle{B}} b_{i_k}),
\end{equation*}
where $\cdot_{\scriptscriptstyle{B}}$ stands for the algebra product in $B$.
This linear form is further extended to a linear form $\Phi\colon T(V)\mapsto\mathbb C$ by
\begin{equation*}
	\Phi(w_1| \cdots |w_p):=\phi(w_1) \cdots \phi(w_p).
\end{equation*}
Notice that the linear form $\Phi \in \mathcal{G}(\mathbb C)$.

\begin{Remark}
All our results would of course hold for $R={\mathbb C}\langle\langle x_1,\ldots, x_n\rangle\rangle$ and a finite family $b_1,\ldots, b_n$ of elements of $B$. However, as handling the countable case does not present any extra difficulty, we state our results in that case and specialize them to the finite setting when appropriate.
\end{Remark}

Recalling \eqref{grell}, the series $\Lambda_{\mathrm{gr}}(\Phi)\in G^1\subset{\mathbb C}\langle\langle x_1,x_2,x_3,\dotsc\rangle\rangle$ is by definition the (multivariate) generating series of moments associated to $(b_n)_{n\in\N}$. For example, the coefficient of $x_1^n$ in $\Lambda_{\mathrm{gr}}(\Phi)$ is $\phi\big(b_1^n\big)$, the moment of order $n$ of the random variable $b_1 \in B$ in the sense of non-commutative probability. We will write $M(x_1)$ for the series in ${\mathbb C}\langle\langle x_1\rangle\rangle$ whose coefficients are the same on the~$x_1^n$ as those of $\Lambda_{\mathrm{gr}}(\Phi)$. This amounts to looking at the univariate case.

\subsection{Free probability}\label{ssec:fproba}

It was shown in \cite{ebrahimipatras_16,ebrahimipatras_17} (to which we refer for details) that the fixed point equation
\begin{equation}
	\label{leftshuffleSeries}
	\Phi = \varepsilon_{\mathbb C} + \kappa \prec \Phi
\end{equation}
defines an infinitesimal character $\kappa\in \mathcal{L}(\mathbb C)$ that corresponds to multivariate free cumulants. That is, in the language of the present article, $\Lambda_{\mathrm{Lie}}(\kappa)\in G^0$ is the multivariate generating series of free cumulants associated to $(b_n)_{n\in\N}$; again, recall \eqref{grell}.

Let us set from now on $\hat\mu :=\Lambda(\mu)\in{\mathbb C}\langle\langle x_1,x_2,x_3,\dotsc\rangle\rangle$ for $\mu\colon T(V)\to\mathbb C$ an arbitrary linear form, recall \eqref{bij}.

\begin{Proposition}\label{appli1}
We have the functional multivariate free moment-cumulant relation
\begin{equation}
	\label{freeMCrel}
	\widehat{\Phi}(x) = 1 + \widehat{\kappa}\big(x\, \widehat{\Phi}(x)\big).
\end{equation}
\end{Proposition}

\begin{proof}
This follows from \eqref{leftshuffleSeries} and \eqref{<1}.
\end{proof}
Now
\[
	\widehat{\Phi}(x) = 1+\sum_{w\in\N^*} m(w)x_w,
	\qquad
	\widehat{\kappa}(x)=\sum_{w\in\N^*} \kappa(w)x_w,
\]
where $m(w)=\varphi(b_{i_1}\cdot_{\scriptscriptstyle{B}} \cdots \cdot_{\scriptscriptstyle{B}} b_{i_k})$ and \(\kappa(w)=k(b_{i_1},\ldots,b_{i_k})\) is the multivariate free cumulant, for the word $w={i_1}\cdots {i_k}$.
The functional multivariate free moment-cumulant relation then becomes
\begin{equation*}
	\widehat{\Phi}(x)
	=1+
	\sum_{k \ge 1}
	\sum_{\substack{u_1,\dotsc,u_k \in \N^*\\v =i_1\dotsm i_k \in \N^k}}
	\kappa(v)m({u_1})\dotsm m({u_k})x_{i_1}x_{u_1}\dotsm x_{i_k}x_{u_k}.
\end{equation*}
This statement implies the well-known free multivariate moment-cumu\-lant relations expressed in terms of non-crossing partitions \cite{speichernica}
\begin{equation*}
	m(w) = \sum_{\pi \in NC(|w|)} \prod_{\pi_i \in \pi} k(b_{i_1},\ldots,b_{i_n}|\pi_i).
\end{equation*}
Here $k(b_{i_1},\ldots,b_{i_n}|\pi_i):=k(b_{i_{j_1}}, \ldots, b_{i_{j_p}})$ for the block $\pi_i=\{j_1 < \cdots <j_p\} \subset [n]$ of the non-crossing partition $\pi \in NC(|w|)$.

\subsection{Boolean probability}\label{ssec:bproba}

The notation used in this section is the same as in the previous one. It was shown in \cite{ebrahimipatras_16,ebrahimipatras_17} (to which we refer for details) that the fixed point equation
\begin{equation}
\label{rshuboo}
	\Phi = \varepsilon_{\mathbb C} + \Phi \succ \beta,
\end{equation}
defines an infinitesimal character $\beta\in \mathcal{L}(\mathbb C)$ that corresponds to multivariate Boolean cumulants. That is, in the language of the present article, $\Lambda_{\mathrm{Lie}}(\beta)$ is the multivariate generating series of Boolean cumulants associated to $(b_n)_{n\in\N}$.

Applying $\Lambda_{\mathrm{gr}}$ to the identity \eqref{rshuboo} yields by \eqref{>1} the multivariate functional Boolean moment-cumulant relation
\begin{equation}
\label{fuboomc}
	\widehat{\Phi}(x) = 1 + \widehat{\beta}(x)\widehat{\Phi}(x).
\end{equation}
Note that the summation on the righthand side of \eqref{>1} simplifies drastically because $\beta \in \mathcal{L}(\mathbb C)$ linearizes the right half-coproduct \eqref{deltaright}. More explicitly, the Boolean moment-cumulant relation reads
\begin{equation*}
	\widehat{\Phi}(x) = 1+ \sum_{w} \sum_{\substack{w=uv \\u \neq w}} m(u)\beta(v) x_w.
\end{equation*}
Identity \eqref{fuboomc} rewrites
\begin{equation*}
	1-\widehat{\beta}(x)= \frac{1}{\widehat{\Phi}(x) }.
\end{equation*}

Let us exemplify how to relate such identities with computations in the group $\mathcal{G}(\mathbb C)$ of Hopf algebra characters.
Theorem~\ref{theorem:monotone} has rather interesting implications. Indeed, let $\Phi \in \mathcal{G}(\mathbb C)$ and consider the image of $\Phi * \Phi^{-1}=\varepsilon_{\mathbb C}$
\begin{equation*}
	\Lambda_{\mathrm{gr}}\big(\Phi * \Phi^{-1}\big)(x)
	= \Lambda_{\mathrm{gr}}(\Phi)\bullet \Lambda_{\mathrm{gr}}\big(\Phi^{-1}\big)(x)
	=\widehat{\Phi^{-1}}(x)\widehat{\Phi}\big(x{\widehat{\Phi^{-1}}}(x)\big).
\end{equation*}
This yields
\begin{equation*}
	\widehat{\Phi^{-1}}(x)\widehat{\Phi}\big(x{\widehat{\Phi^{-1}}}(x)\big)=1.
\end{equation*}
From $\Lambda_{\mathrm{gr}}(\Phi^{-1} * \Phi)(x) =1$, on the other hand, we obtain instead
\begin{equation*}
	 \widehat{\Phi}(x) \widehat{\Phi^{-1}}\big(x{\widehat{\Phi}}(x)\big)=1,
\end{equation*}
which implies that
\begin{equation}
\label{inverse3}
	\widehat{\Phi^{-1}}\big(x{\widehat{\Phi}}(x)\big)=\frac{1}{\widehat{\Phi}(x)}
\end{equation}
in the sense of generating series. In particular,

\begin{Proposition}\label{prop:dict}\quad
\begin{itemize}\itemsep=0pt
\item[$(i)$] For the multivariate generating series of free cumulants, $\widehat{\kappa}(x)$, we have
\begin{equation}
\label{inverse3x}
	\widehat{\Phi^{-1}}(x)=\frac{1}{1+ \widehat{\kappa}(x)}.
\end{equation}
\item[$(ii)$] For the multivariate generating series of Boolean cumulants, we have
\begin{equation}
\label{boolCMrel}
	\widehat\beta(x)=1-\widehat{\Phi^{-1}}\big(x{\widehat{\Phi}}(x)\big).
\end{equation}
\end{itemize}
\end{Proposition}

\begin{proof}
Identity \eqref{inverse3x} follows from \eqref{inverse3} and \eqref{freeMCrel} upon composition with the compositional inverse, $\big(x\widehat{\Phi}(x)\big)^{\langle -1\rangle}$. We underline that \eqref{inverse3x} expresses the inverse of the Hopf algebra character~$\Phi \in \mathcal{G}(\mathbb C)$ in the group $G^1$. Identity \eqref{boolCMrel} is a consequence of \eqref{inverse3}.
\end{proof}

\subsection{Monotone probability}\label{ssec:monotoneproba}

It was shown in \cite{ebrahimipatras_17} (to which we refer for details) that
\begin{equation*}
	\Phi=\exp^*(\rho)
\end{equation*}
defines an infinitesimal character $\rho\in \mathcal{L}(\mathbb C)$ that corresponds to multivariate monotone cumulants. That is, in the language of the present article, $\Lambda_{\mathrm{Lie}}(\rho)$ is the multivariate generating series of monotone cumulants associated to $(b_n)_{n\in\N}$.

Now, introduce a formal parameter \(t\). Define \(\Phi_t:=\exp^*(t\rho)\), and observe that it defines a~1-parameter semigroup, since \(\Phi_t *\Phi_s=\Phi_{t+s}\) and \(\Phi_0=\varepsilon_{\mathbb C}\). Formally taking a~derivative we arrive at the equation
\[
	\dot\Phi_t=\rho\ast\Phi_t=\Phi_t*\rho.
\]
Using \(\Lambda_{\mathrm{gr}}\) and defining \(M_t:=\Lambda_{\mathrm{gr}}(\Phi_t) \in G^1$, $h=\Lambda_{\mathrm{Lie}}(\rho) \in G^0\), we arrive by
using Propositions~\ref{prop:gprod} and~\ref{prop:lprod} at the equations
\[
	\dot M_t(x)=M_t(x)h(xM_t(x))=h(x)+((M_t-1)\triangleleft h)(x).
\]
The first equation is present in \cite[Theorem 6.3]{HS2011a} and \cite[equation~(4.10)]{AHLV2015}. The second equation leads to the expansion:
\begin{equation}
	\label{eq:Mtoh}
	M_t
	=1+th+(h\triangleleft h)\frac{t^2}{2}+((h\triangleleft h)\triangleleft h)\frac{t^3}{6}+\dotsb
	= 1+\sum_{n=1}^\infty R^{(n-1)}_{\triangleleft h}(h)\frac{t^n}{n!}.
\end{equation}
Here, $R^{(n)}_{\triangleleft h}(h):= \big(R^{(n-1)}_{\triangleleft h}(h)\big)\triangleleft h$ with $R^{(0)}_{\triangleleft h}(h):=h$.
Consider now for simplicity the univariate case and expand \(M_t\) as a power series in \(x\), i.e.,
\[
	M_t=1+\sum_{n=1}^\infty m_n(t)x^n.
\]
We can perform some explicit computations using equation~\eqref{eq:prelie.x}: if \(\widehat{\rho}(x)=:h(x)=\sum_{n\ge 1} h_nx^n\) is the generating series of the monotone cumulants, then
\begin{align*}
	R^{(n-1)}_{\triangleleft h}(h)&=\sum_{k=n}^\infty\bigg( \sum_{i_1+\dotsm+i_n=k}(i_1+1)(i_1+i_2+1)\dotsm (i_1+\dotsb+i_{n-1}+1)h_{i_1}\dotsm h_{i_n}\bigg)x^k.
\end{align*}
Therefore, by matching terms in equation~\eqref{eq:Mtoh}, we see that
\[
m_n(t)=\sum_{k=1}^n\sum_{i_1+\dotsb+i_k=n}(i_1+1)\dotsm(i_1+\dotsb+i_{k-1}+1)h_{i_1}\dotsm h_{i_k}\frac{t^k}{k!}.
\]
In low degrees:
\begin{align*}
	&m_1(t)= h_1t,\\
	&m_2(t)= h_2t+h_1^2t^2,\\
	&m_3(t)= h_3t+5h_1h_2\frac{t^2}{2}+h_1^3t^3,\\
	&m_4(t)= h_4t+\left( 3h_1h_3+\frac{3}{2}h_2^2 \right)t^2+\frac{13}{3}h_1^2h_2t^3+h_1^4t^4.
\end{align*}
For \(t=1\) the above formula coincides with \cite[equation~(6.9)]{hasebesaigo_11}, see also \cite[Theorem~2]{ebrahimipatras_17}.

\begin{Remark}
\label{rmk:preLieExp}
As $\sum_{n=1}^\infty R^{(n-1)}_{\triangleleft h}(h)\frac{1}{n!}$ is, by definition, the image of $h$ under the Agrachev--Gamkrelidze operator or ``pre-Lie exponential'' of $h$ (see \cite[Section 6.6]{CP} for details), one can formally lift its computation to the free pre-Lie algebra over a generator $\Forest{[]}$. Using the Chapoton--Livernet basis of non-planar rooted trees for the latter, the coefficient of a tree $\tau$ in the expansion of the pre-Lie exponential of $\bullet$ is known to be the corresponding Connes--Moscovici coefficient~$cm(\tau)$. See \cite{Brouder00} for details and also for an explanation of the terminology and notation.
We obtain
\begin{align*}
&	M_t	= \exp^{\triangleleft}(th)
	=1+ \sum_{\tau\in\mathcal T}\operatorname{cm}(\tau)P_h(\tau)\frac{t^{|\tau|}}{|\tau|!}
		 =1+ \sum_{\tau\in\mathcal T}\frac{1}{\tau!\sigma(\tau)}P_h(\tau)t^{|\tau|},
\end{align*}
where \(P_h\colon \mathcal T\to G^0\) is the unique pre-Lie morphism such that \(P_h(\Forest{[]})=h\).
For example
\[
	P_h\big(\Forest{[[]]}\big)=h\triangleleft h,
	\qquad\
	P_h\big(\Forest{[[][]]}\big)=(h\triangleleft h)\triangleleft h-h\triangleleft(h\triangleleft h).
\]
Recall that $h=h(x)=\sum_{n\ge 1} h_nx^n$. Here, $\tau!$ and $\sigma(\tau)$ are respectively the so-called tree factorial and the symmetry factor of the rooted tree $\tau \in\mathcal T$, both are defined inductively. We refer to \cite{Brouder00} for details.
\end{Remark}

\section{Conclusion}
\label{ssec:conclusion}

In this paper, we have established a dictionary between the shuffle Hopf algebra formulation of moment-cumulant relations in non-commutative probability and the classical approach based on non-commutative formal power series (Theorem~\ref{theorem:monotone}). It is based on identifying a new left-linear group law on the set of non-commutative formal power series with unit constant term (Proposition~\ref{thm:monogroup}). We also identify a (right) pre-Lie law on the latter, which follows from right-linearization of the aforementioned group law. For example, the dictionary identifies the shuffle convolution inverse with the reciprocal of the unit-shifted $R$-transform (Proposition~\ref{prop:dict}). This is particularly interesting as we used the group-inverse of the moment character to construct Wick polynomials (see \cite{EFPTZ_18, EFPTZ2021} for details). The dictionary also permits to describe the monotone moment-cumulant relations as a pre-Lie exponential in non-commutative formal power series (Remark~\ref{rmk:preLieExp}). The dictionary shows that both the shuffle Hopf algebra and the formal power series approaches are tightly related. The former, however, seems to add new perspectives in the understanding of computational and conceptional aspects in the combinatorial approach to non-commutative probability theory.

\subsection*{Acknowledgements} This work was partially supported by the project ``Pure Mathematics in Norway'', funded by the Trond Mohn Foundation and the Troms{\o} Research Foundation. KEF was supported by the Research Council of Norway through project 302831 ``Computational Dynamics and Stochastics on Manifolds'' (CODYSMA). NT was funded by the Deutsche Forschungsgemeinschaft (DFG, German Research Foundation) under Germany's Excellence Strategy – The Berlin Mathematics Research Center MATH+ (EXC-2046/1, project ID: 390685689). KEF and NT would also like to thank the Department of Mathematics at the Saarland University for warm hospitality during a sabbatical visit. FP acknowledges support from the European Research Council (ERC) under the European Union’s Horizon 2020 research and innovation program (Duall project, grant agreement No.~670624), from the ANR project Algebraic Combinatorics, Renormalization, Free probability and Operads – CARPLO (Project-ANR-20-CE40-0007) and from the ANR -- FWF project PAGCAP.

\pdfbookmark[1]{References}{ref}
\LastPageEnding

\end{document}